\documentclass[preprint,12pt]{elsarticle}

\usepackage{amssymb}
\usepackage{amsthm}

\usepackage{amsmath}

\usepackage{nccbbb}%% \bbbonee
\usepackage{hyperref}

\newdefinition{definition}{Definition}[section]
\newdefinition{example}[definition]{Example}
\newdefinition{remark}[definition]{Remark}

\newtheorem{lemma}[definition]{Lemma}
\newtheorem{theorem}[definition]{Theorem}
\newtheorem{corollary}[definition]{Corollary}
\newtheorem{proposition}[definition]{Proposition}

\newcommand{\E}{\mathcal{E}} % sigma-algebra of state space
\newcommand{\F}{\mathcal{F}} % filtration
\newcommand{\A}{\mathcal{A}} % generator
\newcommand{\D}{\mathcal{D}} % domain of generator
\newcommand{\R}{\mathbb{R}} % space of real numbers
\newcommand{\Prb}{\mathbb{P}} % probability
\newcommand{\Ep}{\mathbb{E}} % expectation
\newcommand{\ud}{\mathrm{d}} % differential
\newcommand{\sexp}{\mathrm{Sexp}} % the Stieltjes exponential

\numberwithin{equation}{section}

\journal{XXX}

\begin{document}

\begin{frontmatter}



\title{Exponential Change of Measure for General Piecewise Deterministic Markov Processes\tnoteref{t1}}
\tnotetext[t1]{This work was supported by National Natural Science Foundation of China (11471218) and Hebei Higher School Science and Technology Research Projects (ZD20131017)}


\author[csu]{Zhaoyang Liu}\ead{liu.zhy@csu.edu.cn}

\author[csu]{Yuying Liu}\ead{yuyingliu@csu.edu.cn}

\author[stdu]{Guoxin Liu\corref{cor1}}\ead{liugx@stdu.edu.cn}

\cortext[cor1]{Corresponding author}

\address[csu]{School of Mathematics and Statistics, Central South University, Changsha, Hunan, China, 410083}

\address[stdu]{Department of Applied Mathematics and Physics, Shijiazhuang Tiedao University, Shijiazhuang, Hebei, China, 050043}

\begin{abstract}
  We consider a general piecewise deterministic Markov process (PDMP) $X=\{X_t\}_{t\geqslant 0}$ with measure-valued generator $\A$, for which the conditional distribution function of the inter-occurrence time is not necessarily absolutely continuous. A general form of the exponential martingales is presented as
  \begin{equation*}%\label{eq.M^f.general}
    M^f_t=\frac{f(X_t)}{f(X_0)}\left[\sexp\left(\int_{(0,t]}
    \frac{\ud L(\A f)_s}{f(X_{s-})}\right)\right]^{-1}.
  \end{equation*}
  Using this exponential martingale as a likelihood ratio process, we define a new probability measure. It is shown that the original process remains a general PDMP under the new probability measure. And we find the new measure-valued generator and its domain.
\end{abstract}


\end{frontmatter}


\section{Introduction}
%\label{}

The aim of this paper is to give a detail account of change of probability measure technique for general piecewise deterministic Markov processes based on the measure-valued generator theory proposed by Liu et al. \cite{liu2017measure}.

Piecewise deterministic Markov processes (PDMPs) are introduced by Davis \cite{davis1984piecewise,davis1993markov}.
Jacod and Skorokhod \cite{jacod1996jumping} and Liu et al. \cite{liu2017measure} generalize the concept of PDMPs in the same way. Roughly speaking, a strong Markov process with natural filtration of discrete type is called a \emph{general} PDMP.
%PDMPs are a general class of non-diffusion Markov processes, which involve deterministic motion punctuated by random jumps.
%Let $X=\{X_t\}_{t\geqslant 0}$ be the process, and $\tau_i$ the time when the $i$-th random jump happens. The inter-occurrence times of the jumps $\sigma_i=\tau_i-\tau_{i-1}$ are a sequence of positive and independent random variables with conditional tail distribution $F$.
Starting from a state $x$, the motion of the process $X$ follows a deterministic path of a semi-dynamic system (SDS) $\phi$, i.e., $X_t=\phi_x(t)$, until the first random jump time $\tau_1$. The post-jump location $X_{\tau_1}$ is selected by a transition kernel $q$, and the motion of the process restarts from this new state $X_{\tau_1}$.
In \cite{davis1993markov}, the path of $\phi$ is supposed to be absolutely continuous with respect to time. This assumption is relaxed in \cite{liu2017measure}. It is only assumed to be right-continuous.
%that is, there may be some deterministic jumps between any two adjacent random jumps. So the authors introduce an auxiliary process $X^-$ (defined as (\ref{eq.X-})) which is a predictable version of the original process to distinguish these two kinds of jumps.
In \cite{davis1993markov}, the first random jump occurs either at a random time with a Poisson-like jump rate $\lambda(\phi_x(t))$ or when the SDS $\phi_x(t)$ hits the boundary of the state space. While, for a general PDMP, the tail distribution function of the first random jump time $\tau_1$ is supposed to be general with memorylessness along the path of $\phi$. In this way, these two kinds of random jumps can be dealt with in a unified framework. And it makes the model more general to cover a larger range of entities.
A general PDMP is uniquely determined by the three characteristics $\phi$, $F$ and $q$, so $(\phi,F,q)$ is called the \emph{characteristic triple} of the process.
In \cite{liu2017measure} the authors introduce a new concept of generator called \emph{measure-valued generator} for the general PDMPs. A measure-valued generator $\A$ is a mapping from a measurable function $f$ to an additive function $\A f(x,t)$ of the SDS $\phi$. %, which can reflect not only the absolutely continuous variation but also the singularly continuous and purely discontinuous variation.
The domain of generator is extended from the absolutely path-continuous functions to the locally path-finite variation functions.

%which is not necessarily continuous on $\R_+$ in this paper. That is, $F$ is allowed to have absolutely continuous, singularly continuous and pure discrete parts. The evolution between any two consecutive jumps is determined by a semi-dynamic system (SDS). The trajectories of a semi-dynamic system are not necessarily continuous either, meaning that the system is allowed to have deterministic jumps. To distinguish the PDMP introduced by \cite{davis1984piecewise}, the process described above is called a \emph{general} PDMP throughout this paper.

Exponential change of measure is a useful technique applied in many areas, and has been studied extensively.
It\^o and Watanabe \cite{Ito1965Transformation}, Kunita \cite{Kunita1969Absolute,Kunita1976Absolute} and Palmowski and Rolski \cite{palmowski2002technique} discuss change of measure for general classes of Markov processes.
%\cite{Kunita1969Absolute} characterizes the class of all the Markov process which is absolutely continuous with respect to a given Markov process.
A discussion for L\'evy processes can be found in Sato \cite[Section 33]{Sato1999Levy}.
Cheridito et al. \cite{Cheridito2005Equivalent} discuss this topic for jump-diffusion processes.
Especially, Palmowski and Rolski \cite{palmowski2002technique} present a detail account of change of probability measure technique for c\`adl\`ag Markov processes including the PDMPs in Davis' sense. This technique is widely used for ruin probability (\cite{Hipp2004Asymptotics}), large deviation (\cite{Chetrite2015Nonequilibrium}), derivative pricing (\cite{Bielecki2008A}, \cite{Bo2011Exponential}, \cite{Jiang2008On}, \cite{Shen2013Longevity,Shen2013Pricing}). %And Bo (2011) applies this technique to study the term structures of interest rate and exchange rate.
In \cite{palmowski2002technique}, the exponential martingale used as the likelihood ratio process for change of probability measure is expressed in terms of the extended generator $\hat{\A}$ as follow
\begin{equation}\label{eq.M^f.P&R}
  M^f_t=\frac{f(X_t)}{f(X_0)}\exp\left[-\int_0^t\frac{\hat{\A} f(X_s)}{f(X_s)}\ud s\right].
\end{equation}
However, this expression is not suitable for general PDMPs, and the function $f$ is limited in the domain of the extended generator. We generalize this result for general PDMPs by describing this exponential martingale in terms of measure-valued generator, i.e.,
\begin{equation}\label{eq.M^f.general}
  M^f_t=\frac{f(X_t)}{f(X_0)}\left[\sexp\left(\int_{(0,t]}
  \frac{\ud L(\A f)_s}{f(X_{s-})}\right)\right]^{-1},
\end{equation}
where the operator $L$ is defined as (\ref{eq.L(a)}) and $\sexp$ is the Stieltjes exponential (see (\ref{eq.M^f}) and Remark \ref{rm.Stieltjes.exp} for the details).

The structure of this paper is organized as follows.
In Section 2, we recall some results of \cite{liu2017measure}, including the notations and some basic properties of general PDMPs and the concept of the measure-valued generators for this kind of processes. Inspired by \cite{palmowski2002technique}, we present the expression of exponential martingale for a general PDMP in Section 3. This expression is described in terms of the measure-valued generator. Moreover, in some special cases of PDMPs, this one degenerates into the form of (\ref{eq.M^f.P&R}). And Corollary \ref{cor.M^f.P&R} shows us that (\ref{eq.M^f.P&R}) is suitable for PDMPs not only in Davis' sense. In Section 5, using the exponential martingale as the likelihood ratio process, we give the detail account of exponential change of measure technique for general PDMPs. We show that a general PDMP remains a general PDMP under the new probability measure, and we find its new characteristic triple (see Theorem \ref{thm.new.Markov&triple}). The new measure-valued generator and its domain are given in Theorem \ref{thm.new.generator}. And this new measure-valued generator can also be rewritten in terms of the old one or using the \emph{op\'erateur carr\'e du champ} (see Corollary \ref{cor.new.Af-1&2}).

\section{Preliminary}

\subsection{Definition of a general PDMP}

Let $(E,\E)$ be a Borel space.
We consider an $E$-valued general PDMP $X=\{X_t\}_{0\leqslant t<\tau}$ with life time $\tau$ defined on $(\Omega,\F,\Prb)$.
The jump times are a sequence of stopping times $\{\tau_n\}_{n\geqslant 0}$ satisfying that
\begin{equation*}\label{eq.jumping times}
  \tau_0=0,\quad \tau_{n+1}=\tau_n+\theta_{\tau_n}\circ\tau_1,\quad \tau_n\uparrow\tau,
\end{equation*}
where $\theta_t$ is a shift operator.

The process starts from $x\in E$ and follows the \emph{semi-dynamic system} (SDS) $\phi$ until the first jump time $\tau_1$, i.e., $X_t=\phi_x(t)$ for $t<\tau_1$, where $\phi$ satisfies that
\begin{equation}\label{eq.SDS}
  \phi_x(0)=x,\quad\phi_x(s+t)=\phi_{\phi_x(s)}(t),
\end{equation}
and that $\phi_x(\cdot)$ is c\`adl\`ag for all $x\in E$.
The first jump time $\tau_1$ has the conditional tail distribution function $F$ defined by $$F(x,t)=\Prb\{\tau_1>t\,|X_0=x\},$$
which is called the \emph{conditional survival function}.
The location of the process at the jump time $\tau_1$ is selected by the \emph{transition kernel} $q$ defined by
$$q(x,t,B)= \Prb\{X_{\tau_1}\in B\,|X_0=x,\tau_1=t\},\quad B\in\E.$$
And then the process restarts from this new state $X_{\tau_1}$ as before.
Thus the process can be expressed as
\begin{equation}\label{eq.PDMP}
  X_t=\sum_{n=0}^\infty\phi_{X_{\tau_n}}(t-\tau_n)\bbbone_{\{\tau_n\leqslant t<\tau_{n+1}\}}.
\end{equation}
$(\phi,F,q)$ is called the \emph{characteristic triple} of the general PDMP $X$.
A general PDMP $X$ is \emph{regular} if $\Prb\{\tau=\infty|X_0=x\}=1$ for every $x\in E$.

Set $$c(x)=\inf\{t>0:F(x,t)=0\},$$ and
\begin{equation*}
  \mathcal{I}_x=
\left\{
  \begin{array}{ll}
    \R_+, & c(x)=\infty; \\
    \displaystyle[0,c(x)), & c(x)<\infty,\,F(x,c(x)-)=0; \\
    \displaystyle[0,c(x)], & c(x)<\infty,\,F(x,c(x)-)>0.
  \end{array}
\right.
\end{equation*}
Note that, for a general PDMP $X$, the conditional survival function $F$ and the transition kernel $q$ satisfy that
\begin{gather}
  F(x,0)=1,\quad F(x,s+t)=F(x,s)F(\phi_x(s),t), \label{eq.F}\\
  q(x,t,\{\phi_x(t)\})=0,\quad q(x,s+t,B)=q(\phi_x(s),t,B),
\end{gather}
for all $s,t\in\R_+$ such that $s+t\in\mathcal{I}_x$.

For convenience, we extend the state space by adding an isolated point $\Delta$ to $E$. The SDS $\phi$ is also extended by
\begin{equation}\label{eq.extend SDS}
  \phi_x(c(x))=
  \left\{
    \begin{array}{ll}
      \displaystyle \lim_{t\uparrow c(x)}\phi_x(t), &
      \hbox{if $\displaystyle \lim_{t\uparrow c(x)}\phi_x(t)$ exists;} \\
      \Delta, & \hbox{otherwise,}
    \end{array}
  \right.
\end{equation}
for $x\in E$. Set
$$\partial^+E=\{\phi_x(c(x)):x\in E\},$$
and denote $\bar{E}=E\cup\partial^+E$.
For any $x\in E$,
$$\Phi_x=\{\phi_x(t):t\in\mathcal{I}_x\},$$
the subset of $\bar{E}$, is called a \emph{trajectory} of SDS $\phi$.

%The essential difference between an SDS and a dynamic system is that, starting from any $x\in E$, an SDS can predict the future, but cannot recall the history exactly. That is,
  For an SDS, different trajectories starting from different states may join together at some states.
  A state $x\in E$ is called a \emph{confluent state} of SDS $\phi$ if for any small $s>0$ there exist two distinguishable states $x_1,x_2$ such that $x=\phi_{x_1}(s)=\phi_{x_2}(s)$.

\subsection{Additive functionals and measure-valued generator}

To study the properties of general PDMPs, Liu et al. \cite{liu2017measure} introduce the so-called \emph{additive functionals} of an SDS.

\begin{definition}
  Let $\phi$ be an SDS. A measurable function $a:E\times\R_+\mapsto\R$ such that $a(x,\cdot)$ is c\`adl\`ag for each $x\in E$ is called an \emph{additive functional} of the SDS $\phi$ if for any $x\in E$, $s,t\in\R_+$ and $s+t\in\mathcal{I}_x$,
  \begin{equation}\label{eq.add of SDS}
    a(x,0)=0,\quad a(x,s)+a(\phi_x(s),t)=a(x,s+t).
  \end{equation}
  An additive functional $a$ is called to have \emph{locally finite variation} if $a(x,\cdot)$ has locally finite variation on $\mathcal{I}_x$ for all $x\in E$, i.e.,
  \begin{equation*}
    \int_{(0,t]}\big|a\big|(x,\ud s)<\infty \quad \hbox{for all }x\in E,\,t\in\mathcal{I}_x.
  \end{equation*}
  The space of all the additive functionals of the SDS $\phi$ is denoted by $\mathfrak{A}_\phi$. And denote by $\mathfrak{A}_\phi^{loc}$, the space of all the additive functionals of the SDS $\phi$ with locally finite variation.
\end{definition}

The Lebesgue decomposition of an additive functional of the SDS $\phi$ is given in \cite{liu2017measure} as follow.
Denote
$$J_a=\{\phi_x(t):a(x,t)-a(x,t-)\neq 0,x\in E,t\in\mathcal{I}_x\setminus\{0\}\},$$
which consists all the jumping states of $a(x,\cdot)$ for each $x\in E$.

\begin{theorem}\label{thm.add Lebesgue}
  %Let $a$ be an additive functional of the SDS $\phi$ satisfying that $a(x,\cdot)$ is of locally finite variation on $\mathcal{I}_x$ for any $x\in E$.
  Let $a\in \mathfrak{A}_\phi^{loc}$.
  Assume that $J_a$ contains no confluent state. Then, for any $x\in E$, there exist measurable functions $\mathcal{X}a$ and $\Delta a$ with
$$\int_0^t\big|\mathcal{X}a(\phi_x(s))\big|\ud s<\infty \quad\hbox{and}\quad
\sum_{0<s\leqslant t}\big|\Delta a(\phi_x(s))\big|<\infty\quad \hbox{for all } t\in\mathcal{I}_x$$
such that $a(x,\cdot)$ has the Lebesgue decomposition
\begin{equation}\label{eq.a Lebesgue}
  a(x,t)=\int_0^t\mathcal{X}a(\phi_x(s))\ud s+a^{sc}(x,t)+\!\sum_{0<s\leqslant t}\!\Delta a(\phi_x(s)),
  \quad t\in\mathcal{I}_x,
\end{equation}
%\begin{equation}\label{eq.a Lebesgue}
%  a(x,t)=a^{ac}(x,t)+a^{sc}(x,t)+a^{pd}(x,t),\quad
%t\in\mathcal{I}_x,
%\end{equation}
where
%$$a^{ac}(x,t)=\int_0^t\mathcal{X}a(\phi_x(s))\ud s,\quad
%a^{pd}(x,t)=\sum_{0<s\leqslant t}\Delta a(\phi_x(s)),$$
\begin{align*}%\label{}
  &\mathcal{X}a(x)=
\left\{
  \begin{array}{ll}
    \displaystyle\frac{\partial^+a(x,t)}{\partial t}\Big|_{t=0}, &
\hbox{if $\displaystyle\frac{\partial^+a(x,t)}{\partial t}\Big|_{t=0}$ exists;} \\
    0, & \hbox{otherwise,}
  \end{array}
\right.\\%\quad x\in E, \\
  &\Delta a(\phi_x(t))=a(x,t)-a(x,t-),%\quad\quad\quad
%x\in E,\,t\in\mathcal{I}_x\setminus\{0\}.
\end{align*}
and $a^{sc}(x,\cdot)$ is the singularly continuous part of $a(x,\cdot)$. Moreover, the three terms on the right side of \textnormal{(\ref{eq.a Lebesgue})} are all additive functionals of SDS $\phi$.
\end{theorem}

The function $\Lambda$ defined by
\begin{equation*}
  \Lambda(x,t)=\int_{(0,t]}\frac{F(x,\ud s)}{F(x,s-)},\quad x\in E,\,t\in\mathcal{I}_x
\end{equation*}
is called the \emph{conditional hazard function}. It follows from (\ref{eq.F}) that $\Lambda$ is an additive functional of the SDS $\phi$. $\Lambda$ and $F$ are uniquely determined by each other. By Theorem \ref{thm.add Lebesgue}, $\Lambda$ has the Lebesgue decomposition
\begin{equation*}
  \Lambda(x,t)=\int_0^t\lambda(\phi_x(s))\ud s+\Lambda^{sc}(x,t)+\sum_{0<s\leqslant t} \Delta\Lambda(\phi_x(s))
\end{equation*}
for $x\in E$, $t\in\mathcal{I}_x$. Here we denote $\lambda=\mathcal{X}\Lambda$.

The transition kernel can be simplified in some circumstance. \cite{liu2017measure} proves the following lemma.
\begin{lemma}\label{lem.q=Q}
  %If $J_\Lambda$ contains no confluent state, then
  There exists a measurable function $Q:E\times\E\mapsto[0,1]$ such that, for $x\in E$ and $t\in\mathcal{I}_x\setminus\{0\}$, if $\phi_x(t)$ is not a confluent state, then
\begin{equation*}
  q(x,t,B)=Q(\phi_x(t),B),\quad B\in\E.
\end{equation*}
%for all $x\in E$, $t\in\mathcal{I}_x\setminus\{0\}$ and $B\in\E$.
\end{lemma}

Throughout this paper, we assume that $J_\Lambda$ contains no confluent state. Thus, $(\phi,\Lambda,Q)$ is also referred as the characteristic triple of the general PDMP $X$.

%Let $a$ be an additive functional of the SDS $\phi$.
For $a\in \mathfrak{A}_\phi$, define an operator $L$ such that $L(a)=\{L(a)_t\}_{0\leqslant t<\tau}$ is a process satisfying that
\begin{equation}\label{eq.L(a)}
  \begin{array}{l}
    L(a)_0=0, \\
    L(a)_t=L(a)_{\tau_n}+a(X_{\tau_n},t-\tau_n),\quad
\tau_n< t\leqslant\tau_{n+1},\;n=0,1,2,\dots
  \end{array}
\end{equation}
%\begin{align*}\label{eq.L(a)}
%  L(a)_0&=0,\nonumber\\
%  L(a)_t&=L(a)_{\tau_n}+a(X_{\tau_n},t-\tau_n),\quad
%\tau_n< t\leqslant\tau_{n+1},\,n=0,1,2,\dots
%\end{align*}
According to \cite{liu2017measure}, $L(a)$ is a predictable additive functional of the general PDMP $X$.

Throughout this paper we denote the space of measurable functions $f:\bar{E}\mapsto\R$ by $\mathcal{M}(\bar{E})$, and add a subscript $b$ to denote to restriction to bounded functions.
For a general PDMP $X$, Liu et al.\cite{liu2017measure} presents a new form of generator called the \emph{measure-valued generator} $\A$ such that $\A f\in \mathfrak{A}_\phi$. % is an additive functional of the SDS $\phi$.
Notice that $\A f(x,\cdot)$ is a signed measure on $\mathcal{I}_x$ for any fixed $x\in E$. That is why we call it `measure-valued'. $\D(\A)$ denotes the \emph{domain} of the measure-valued generator $\A$ which consists all the functions $f\in\mathcal{M}(\bar{E})$
%such that $L^{\A f}$ has locally finite variation.
%$\int_{(0,t]}|\ud L^{\A f}_s|<\infty\;\Prb$-a.s. for all $t\geqslant 0$.
%Or equivalently, $\D(\A)$ consists all the functions $f\in\mathcal{M}(\bar{E})$
satisfying:
\begin{enumerate}[(i)]
  \item $f$ is of \emph{locally path-finite-variation}, i.e., $f(\phi_x(\cdot))$ is of finite variation on any closed subinterval of $\mathcal{I}_x$ for any $x\in E$;
  \item for any $x\in E,\,t\in\mathcal{I}_x$ we have
  \begin{equation}\label{eq.D(A) condition}
  \int_{(0,t]}\int_E\big|f(y)-f(\phi_x(s))\big|Q(\phi_x(s),\ud y)\Lambda(x,\ud s)<\infty.
  \end{equation}
\end{enumerate}
Thus
\begin{equation}\label{eq.U^f}
  U^f_t=f(X_t)-\int_{(0,t]}\ud L(\A f)_s,\quad 0\leqslant t<\tau
\end{equation}
is a $\Prb$-local martingale prior to $\tau$ for $f\in\D(\A)$.
Especially, if
\begin{equation}\label{eq.martingale condition}
  \Ep_x\left[\sum_{n=1}^{\infty}\Big|f(X_{\tau_n})-f(X_{\tau_n}^-)\Big|\right]<\infty,\quad x\in E,\,t\in\R_+,
\end{equation}
$U^f$ is a $\Prb$-martingale.
Moreover, if $f\in\D(\A)$, then for $x\in E$, $t\in\mathcal{I}_x$,
\begin{equation}\label{eq.generator}
  \A f(x,\ud t)=D f(x,\ud t)+\Lambda(x,\ud t)\int_E[f(y)-f(\phi_x(t))]Q(\phi_x(s),\ud y),
\end{equation}
where
\begin{equation*}\label{eq.Df}
  D f(x,t) = f(\phi_x(t))-f(x)
\end{equation*}
is the additive functional of the SDS $\phi$ induced by function $f$.

Since $D f$ and $\Lambda$ are both additive functionals of the SDS $\phi$,
following Theorem \ref{thm.add Lebesgue}, we have the Lebesgue decomposition of $\A f$,
\begin{equation*}
  \A f(x,t)=\int_0^t\mathcal{X}\A f(\phi_x(s))\ud s+\A^{sc} f(x,t)+\sum_{0<s\leqslant t}\Delta\A f(\phi_x(s))
\end{equation*}
for $x\in E$, $t\in\mathcal{I}_x$, where
\begin{align*}
  \mathcal{X}\A f(x) &=  \mathcal{X}f(x)
+\lambda(x)\int_E[f(y)-f(x)]Q(x,\ud y),\quad x\in E,\\
  \Delta\A f(x) &= \Delta f(x)
+\Delta\Lambda(x) \int_E[f(y)-f(x)]Q(x,\ud y),\quad x\in\bar{E},
\end{align*}
and
\begin{equation*}
  \A^{sc}f(x,\ud t)=D^{sc}f(x,\ud t)
+\Lambda^{sc}(x,\ud t)\int_E[f(y)-f(\phi_x(t))]Q(\phi_x(t),\ud y).
\end{equation*}
Here we simply denote $\mathcal{X}f=\mathcal{X}Df$ and $\Delta f=\Delta Df$. And denote
\begin{equation*}
  \A^{ac}f(x,t)=\int_0^t\mathcal{X}\A f(\phi_x(s))\ud s \quad\hbox{and}\quad
  \A^{pd}f(x,t)=\sum_{0<s\leqslant t}\Delta\A f(\phi_x(s))
\end{equation*}
for $x\in E$, $t\in\mathcal{I}_x$.

\section{Exponential martingale}

Consider a regular general PDMP $X=\{X_t\}_{t\geqslant 0}$ defined on a probability space $(\Omega,\F,\Prb)$. % with the measure-valued generator $(\A,\D(\A))$.
Define an auxiliary process $X^-=\{X^-_t\}_{t\geqslant 0}$ by
\begin{equation}\label{eq.X-}
  X_t^-=X_0\bbbone_{[t=0]}+\sum_{n=0}^\infty\phi_{X_{\tau_n}}(t-\tau_n) \bbbone_{\{\tau_n< t\leqslant \tau_{n+1}\}}.
\end{equation}
Obviously, $X_t^-=X_t$ if $t\neq\tau_n$, $n=1,2,\dots$ In other words, $X^-$ modifies the values of $X$ only at random jumping times. And it is easy to check that $X^-$ is a predictable process.

For a linear operator $\A:\mathcal{M}(\bar{E})\mapsto \mathfrak{A}_\phi$ and each strictly positive function $f\in\mathcal{M}(\bar{E})$, we define a process $M^f=\{M^f_t\}_{t\geqslant 0}$ by
\begin{equation}\label{eq.M^f}
  M^f_t=\frac{f(X_t)}{f(X_0)}
\exp\left[-\int_{(0,t]}\frac{\ud L(\A^cf)_s}{f(X_{s-})}\right]
\prod_{0<s\leqslant t}\left[1+\frac{\Delta\A f(X_s^-)}{f(X_{s-})}\right]^{-1},
\end{equation}
where $\A^cf=\A^{ac}f+\A^{sc}f$. And $M^f$ can also be written as
\begin{equation*}
  M^f_t=\frac{f(X_t)}{f(X_0)}
  \exp\left[-\int_{(0,t]}\frac{\ud L(\A^cf)_s}{f(X_{s-})}
  -\sum_{0<s\leqslant t}\log\left[1+\frac{\Delta\A f(X_s^-)}{f(X_{s-})}\right]\right]
\end{equation*}
If, for some function $h$, the process $M^h$ is a martingale, then it is said to be an \emph{exponential martingale}. In this case, we call $h$ a \emph{good} function.

\begin{remark}\label{rm.Stieltjes.exp}
  For every c\`adl\`ag process $A=\{A_t\}_{t\geqslant 0}$ with $A_0=0$ one can decompose $A_t=A_t^c+A_t^{pd}$, where $A^c$ denotes the continuous part of $A$, and $A^{pd}$ stands for the purely discontinuous part.
  We can pathwise define the \emph{Stieltjes exponential} $\sexp$ by
  $$\sexp(A_t)=\exp(A^c_t)\prod_{0<s\leqslant t}(1+A_s-A_{s-}), \quad t\geqslant 0,$$
  which is also called the \emph{stochastic exponential} or \emph{Dol\'eans-Dade exponential} (see \cite{Protter2004Stochastic}).
  In this sense, (\ref{eq.M^f}) is equivalent to
  $$M^f_t=\frac{f(X_t)}{f(X_0)}\left[\sexp\left(\int_{(0,t]}
  \frac{\ud L(\A f)_s}{f(X_{s-})}\right)\right]^{-1}.$$
  This is why we still call it an exponential martingale when it is a martingale.
\end{remark}

Define
\begin{align*}
  \mathcal{M}^*(\A)=\Big\{& f\in\mathcal{M}(\bar{E}):f(x)\neq 0,\;\;
\int_{(0,t]}\big|\A f\big|(x,\ud s)<\infty, \\
&\int_{(0,t]}\frac{\big|\A^cf\big|(x,\ud s)}{\big|f(\phi_x(s-))\big|}<\infty,\;\;
0<\prod_{0<s\leqslant t}\left[1+\frac{\Delta\A f(\phi_x(s))}{f(\phi_x(s-))}\right]<\infty\\
&
\hbox{ for all } x\in E,\,t\in\mathcal{I}_x\Big\}.
\end{align*}
%\begin{align*}
%  \mathcal{M}^*(\A)=\Big\{& f\in\mathcal{M}(\bar{E}):f(x)>0,\;\;
%\int_{(0,t]}\big|\A f(x,\ud s)\big|<\infty, \\
%&\int_{(0,t]}\left|\frac{\A^cf(x,\ud s)}{f(\phi_x(s-))}\right|<\infty,\;\;
%0<\prod_{0<s\leqslant t}\left[1+\frac{\Delta\A f(\phi_x(s))}{f(\phi_x(s-))}\right]<\infty\\
%&
%\hbox{ for all } x\in E,\,t\in\mathcal{I}_x\Big\}.
%\end{align*}
Moreover, for a measure-valued generator $\A$, we let
\begin{equation*}
  \D^*(\A)=\mathcal{M}^*(\A)\cap\D(\A).
\end{equation*}
Thus, $M^f$ defined as (\ref{eq.M^f}) makes sense for $f\in\mathcal{M}^*(\A)$.

The following lemma is an extension of \cite[Proposition 3.2]{Ethier1986Markov} and \cite[Lemma 3.1]{palmowski2002technique}.

\begin{lemma}\label{lem.D&E.lm}
  Let $f\in\mathcal{M}^*(\A)$. Then the process $U^f=\{U^f_t\}_{t\geqslant 0}$ is a local martingale if and only if $M^f=\{M^f_t\}_{t\geqslant 0}$ is a local martingale.
\end{lemma}

\begin{proof}
Assume that the process $U^f$ is a local martingale. For $f\in\mathcal{M}^*(\A)$, denote
\begin{equation*}
  Y^f_t=\exp\left[-\int_{(0,t]}\frac{\ud L(\A^cf)_s}{f(X_{s-})}\right]
\prod_{0<s\leqslant t} \left[1+\frac{\Delta\A f(X_s^-)}{f(X_{s-})}\right]^{-1}.
\end{equation*}
Thus we have
$$\ud Y^f_t=-\frac{Y^f_{t-}}{f(X_{t-})+\Delta\A f(X_t^-)}\ud L(\A f)_t.$$
%and
%$$\ud [f(X_t),Y^f_t]_t=-\frac{Y^f_{t-}}{f(X_{t-})+\Delta\A f(X_t^-)}\ud [f(X_t),L_t^{\A f}]_t,$$
%where
%$$[f(X_t),L_t^{\A f}]_t=\sum_{0<s\leqslant t}\Delta f(X_s)\,\Delta L_s^{\A f}
%=\sum_{0<s\leqslant t}\Delta f(X_s)\,\Delta \A f(X_s^-).$$
By the formula of integration by parts,
\begin{align*}
  \ud M^f_t=&\frac{1}{f(X_0)}\Big\{Y^f_{t-}\ud f(X_t)+f(X_t)\ud Y^f_t\Big\} \\
           =&\frac{Y^f_{t-}}{f(X_0)}\left\{\ud f(X_t)-\frac{f(X_t)\,\ud L(\A f)_t} {f(X_{t-})+\Delta\A f(X_t^-)}\right\} \\
           =&\frac{Y^f_{t-}}{f(X_0)}\left\{\ud f(X_t)+\frac{L(\A f)_{t-}\ud f(X_t)-\ud\big(f(X_t)\, L(\A f)_t\big)}{f(X_{t-})+\Delta\A f(X_t^-)}\right\} \\
           =&\frac{Y^f_{t-}}{f(X_0)} \frac{f(X_{t-})\,\ud f(X_t)+L(\A f)_t\,\ud f(X_t)-\ud\big(f(X_t)\, L(\A f)_t\big)}{f(X_{t-})+\Delta\A f(X_t^-)} \\
           =&\frac{Y^f_{t-}}{f(X_0)} \frac{f(X_{t-})\,\ud f(X_t)-f(X_{t-})\,\ud L(\A f)_t}{f(X_{t-})+\Delta\A f(X_t^-)} \\
           =&\frac{M^f_{t-}}{f(X_{t-})+\Delta\A f(X_t^-)}\ud U^f_t.
\end{align*}
Hence the process $M^f$ is a local martingale.

Now, assume that $M^f$ is a local martingale. We also have
\begin{equation*}
  \ud U^f_t=\frac{f(X_{t-})+\Delta\A f(X_t^-)}{M^f_{t-}}\ud M^f_t.
\end{equation*}
Thus $U^f$ is a local martingale. This completes the proof.
\end{proof}

Now let $\A$ be the measure-valued generator of the general PDMP $X$ with the domain $\D(\A)$. With the terminology in \cite{liu2017measure}, we give the different forms of the exponential martingale $M^f$ and its domain.

\begin{corollary}
  For a quasi-Hunt PDMP $X$, let $f\in\D^*(\A)$.
  \begin{equation}\label{eq.M^f.Hunt}
    M^f_t=\frac{f(X_t)}{f(X_0)}\exp\left[-\int_{(0,t]}\frac{\ud L(\A^cf)_s}{f(X_{s-})}\right],\quad t\geqslant 0
  \end{equation}
  if and only if $f$ is path-continuous, that is, $f(\phi_x(\cdot))$ is continuous on $\mathcal{I}_x$ for each $x\in E$.
\end{corollary}

\begin{proof}
  $X$ is quasi-Hunt, which means that $\Delta\Lambda=0$. Comparing (\ref{eq.M^f}) with (\ref{eq.M^f.Hunt}), we have
  $$\prod_{0<s\leqslant t} \left[1+\frac{\Delta\A f(X_s^-)}{f(X_{s-})}\right]^{-1}= 1\quad \hbox{for all }t\geqslant 0,$$
  which means $\Delta\A f=0$. Then, following from the Lebesgue decomposition of (\ref{eq.generator}), we get $\Delta f=0$, i.e., $f$ is path-continuous.

  Conversely, the path-continuity of $f$ means $\Delta f=0$. Following from the Lebesgue decomposition of the measure-valued generator (\ref{eq.generator}), we have $\Delta\A f=0$. Thus (\ref{eq.M^f}) and (\ref{eq.M^f.Hunt}) are the same in this situation.
\end{proof}

\begin{corollary}\label{cor.M^f.P&R}
  For a general PDMP $X$, let $f\in\D^*(\A)$. If any one of the following conditions holds:
  \begin{enumerate}[(i)]
    \item $\A^{sc}f=\A^{pd}f=0$;
    \item $X$ is quasi-It\^o, and $f$ is absolutely path-continuous;
    \item for any $x\in E,t\in\mathcal{I}_x$, $\Lambda(x,t)=\int_0^t\lambda(\phi_x(s))\ud s+\bbbone_\Gamma(\phi_x(t))$, $f$ is absolutely path-continuous with boundary condition $f(x)=\int_Ef(y)Q(x,\ud y)$ for $x\in\Gamma$, where $\Gamma=\{\phi(c(x),x):F(x,c(x)-)>0,c(x)<\infty,x\in E\}$.
  \end{enumerate}
  Then
  \begin{equation}\label{eq.M^f.Ito}
    M^f_t=\frac{f(X_t)}{f(X_0)}\exp\left[-\int_0^t\frac{\mathcal{X}\A f(X_s)}{f(X_s)}\ud s\right],\quad t\geqslant 0.
  \end{equation}
\end{corollary}

\begin{proof}
  (i) If $\A^{sc}f=\A^{pd}f=0$, we have
  $$\int_{(0,t]}\frac{\ud L(\A^cf)_s}{f(X_{s-})}%=\int_{(0,t]}\frac{\ud L_s^{\A^{ac}f}}{f(X_{s-})}
  =\int_{(0,t]}\frac{\mathcal{X}\A f(X_s)}{f(X_s)}\ud s \quad\hbox{and}\quad
  \prod_{0<s\leqslant t}\left[1+\frac{\Delta\A f(X_s^-)}{f(X_{s-})}\right]^{-1}=1 $$
  for all $t\geqslant 0$. Thus we get (\ref{eq.M^f.Ito}).

  (ii) $X$ is quasi-It\^o, which is equivalent to that $\Lambda=\Lambda^{ac}$. If $f$ is absolutely path-continuous, then $D^{sc}f=D^{pd}f=0$. Therefore, we get $\A^{sc}f=\A^{pd}f=0$.

  (iii) Following the condition, we have $\Lambda^{sc}=0$ and
  $$\Delta\Lambda(x)=\left\{
                       \begin{array}{ll}
                         1, & x\in\Gamma; \\
                         0, & x\in E\setminus\Gamma.
                       \end{array}
                     \right.
  $$
For an absolutely path-continuous function $f$ with $f(x)=\int_Ef(y)Q(x,\ud y)$ for $x\in\Gamma$, we have $D^{sc}f=D^{pd}f=0$. Then, it follows from the Lebesgue decomposition of $\A f$ that $\A^{sc}f=\A^{pd}f=0$.
\end{proof}

Note that the condition (iii) is the case of PDMPs in the sense of \cite{davis1993markov}. And the condition (ii) and (iii) are both special cases of (i).

\begin{corollary}
  For a quasi-step PDMP $X$, let $f\in\D^*(\A)$.
\begin{equation}\label{eq.M^f.step}
  M^f_t=\frac{f(X_t)}{f(X_0)}\prod_{0<s\leqslant t}\left[1+\frac{\Delta\A f(X_s^-)}{f(X_{s-})}\right]^{-1},\quad t\geqslant 0
\end{equation}
if and only if $f$ is a path-step function.
\end{corollary}

\begin{proof}
  For a quasi-step PDMP $X$, $\Lambda=\Lambda^{pd}$. If (\ref{eq.M^f.step}) holds, we have
$$\int_{(0,t]}\frac{\ud L(\A^cf)_s}{f(X_{s-})}=0\quad \hbox{for all }t\geqslant 0,$$
which means $\A^{ac}f=\A^{sc}f=0$. Then, following from the Lebesgue decomposition of $\A f$, we have $D^{ac}f=D^{sc}f=0$, that is, $f$ is a path-step function.

Conversely, if $f$ is a path-step function, then we get $\A^{ac}f=\A^{sc}f=0$ for a quasi-step PDMP $X$. Hence, the proof is completed.
\end{proof}

%\begin{remark}
%Let $(\A',\D(\A'))$ denote the extended generator and its domain for the general PDMP $X$ (see \cite{davis1993markov} Definition 14.15). (Liu et al. (2015)) Theorem 5.7 proves that $\D(\A')$ contains all $f\in\D(\A)$ with the constraint conditions $\A^{sc}f=\A^{pd}f=0$. And $\A'f=\mathcal{X}\A f$ for $f\in\D(\A')$, that is,
%\begin{equation*}
%  \A f(x,t)=\int_0^t\A'f(\phi_x(s))\ud s, \quad x\in E,\,t\in\mathcal{I}_x.
%\end{equation*}
%In this case,
%\begin{equation*}
%  D^f_t=f(X_t)-\int_0^t\A'f(X_s)\ud s,
%\end{equation*}
%and (\ref{eq.M^f}) becomes
%\begin{equation}\label{eq.M^f.extended}
%  M^f_t=\frac{f(X_t)}{f(X_0)}\exp\left[-\int_0^t\frac{\A'f(X_s)}{f(X_{s})}\ud s\right],
%\end{equation}
%which is exactly the form of exponential martingale proposed by \cite{palmowski2002technique}.
%\end{remark}

\begin{proposition}
  For a strictly positive function $h\in\D(\A)$, if $\A h(x,\cdot)$ only has a finite number of discontinuous points on $\mathcal{I}_x$ for every $x\in E$,
  $$h\in\mathcal{M}_b(\bar{E}),\quad\inf_{x\in\bar{E}}b(x)>-1,\quad\sup_{x\in\bar{E}}b(x)<\infty,$$
  and either one of the following two conditions holds:
\begin{enumerate}[(C1)]
  \item $a\in \mathfrak{A}_\phi^{loc}$;
  \item $\A^ch\in \mathfrak{A}_\phi^{loc}$, $\displaystyle\inf_{x\in\bar{E}}h(x)>0$,
\end{enumerate}
where
$$a(x,t)=\int_{(0,t]}\frac{\A^ch(x,\ud s)}{h(\phi_x(s-))},\quad
b(x)=\frac{\Delta\A h(x)}{h(x)-\Delta h(x)},$$
then $h$ is a good function.
\end{proposition}

\begin{proof}
  First we need to show that $h\in\mathcal{M}^*(\A)$.
  By Lemma \ref{lem.D&E.lm}, we know that $M^h$ is a local martingale if $h\in\D^*(\A)=\D(\A)\cap\mathcal{M}^*(\A)$. Then, applying \cite[Theorem 51]{Protter2004Stochastic}, we need to show that $\Ep[\bar{M}^h_t]<\infty$ for every $t\geqslant 0$ where $\bar{M}^h_t=\sup_{s\leqslant t}|M^h_s|$.

  It is obvious that $h(x)>0$ and $\A h\in\mathfrak{A}_\phi^{loc}$ for a strictly positive function $h\in\D(\A)$.
  Since $h\in\mathcal{M}_b(\bar{E})$, there exists $H>0$ such that $|h(x)|<H$ for all $x\in\bar{E}$.

  Let $k(x)$ denote the number of discontinuous points of function $\A h(x,\cdot)$ on $\mathcal{I}_x$, and $K=\max_{x\in E}k(x)<\infty$.
  Following the conditions $B_-=\inf_{x\in\bar{E}}b(x)>-1$ and $B_+=\sup_{x\in\bar{E}}b(x)<\infty$, we have
  %$$0<1+B_1\leqslant 1+b(x)\leqslant 1+B_2<\infty\quad\hbox{for all } x\in E.$$
  $$1+b(x)\in[1+B_-,1+B_+]\subset(0,\infty)\quad\hbox{for all } x\in \bar{E}.$$
  Thus,
  \begin{align*}
    \prod_{0<s\leqslant t}\left[1+\frac{\Delta \A h(\phi_x(s))}{h(\phi_x(s-))}\right]
  =&\prod_{0<s\leqslant t}[1+b(\phi_x(s))]\\
  \in&\left[1\wedge(1+B_-)^K,1\vee(1+B_+)^K\right]\subset(0,\infty)
  \end{align*}
  for all $t\in\mathcal{I}_x$, $x\in E$.
  Furthermore,
  \begin{align*}
    \prod_{0<s\leqslant t}\left[1+\frac{\Delta\A h(X_s^-)}{h(X_{s-})}\right]^{-1}
    =&\prod_{0<s\leqslant t}\left[1+b(X_s^-)\right]^{-1}\\
    \in&\left[1\wedge(1+B_+)^{-KN_t},1\vee(1+B_-)^{-KN_t}\right]\subset(0,\infty)
  \end{align*}
  holds for $\Prb$-a.s. Here notice that the process $X$ is regular, thus $N_t<\infty$ $\Prb$-a.s. for all $t\geqslant 0$.

  First, we assume that the condition (C1) holds. Since $a\in\mathfrak{A}_\phi^{loc}$, we have $h\in\mathcal{M}^*(\A)$, and the process
  $\int_{(0,t]}\frac{\ud L(\A^ch)_s}{h(X_{s-})}$ has finite variation $\Prb$-a.s., i.e., $\int_{(0,t]}\left|\frac{\ud L(\A^ch)_s}{h(X_{s-})}\right|<\infty$ for every $t\geqslant 0$.
  Thus
  \begin{equation*}
    \bar{M}^h_t\leqslant \frac{H}{h(X_0)}\exp\left[\int_{(0,t]}\left|\frac{\ud L(\A^ch)_s}{h(X_{s-})}\right|\right]\left(1\vee(1+B_-)^{-KN_t}\right)<\infty \quad\Prb\hbox{-a.s.}
  \end{equation*}
  Then $\Ep[\bar{M}^h_t]<\infty$ for every $t\geqslant 0$, $M^h$ is a martingale, and $h$ is a good function.

  If the condition (C2) holds. Let $H_-=\inf_{x\in\bar{E}}h(x)>0$. Then we have $\frac{1}{h(x)}\leqslant\frac{1}{H_-}$ for any $x\in E$. And
  $$\int_{(0,t]}\left|\frac{\A^ch(x,\ud s)}{h(\phi_x(s-))}\right|
  \leqslant \frac{1}{H_-}\int_{(0,t]}\big|\A^ch(x,\ud s)\big|
  \leqslant \frac{1}{H_-}\int_{(0,t]}\big|\A h(x,\ud s)\big|<\infty,$$
  which means $a\in\mathfrak{A}_\phi^{loc}$. The conclusion can be got by condition (C1).
%  \begin{align*}
%    \int_{(0,t]}\left|\frac{\ud L(\A^ch)_s}{h(X_{s-})}\right|
%    \leqslant & \frac{1}{H_-}\int_{(0,t]}\big|\ud L(\A^ch)_s\big| \\
%    \leqslant & \frac{1}{H_-}\int_{(0,t]}\big|\ud L(\A h)_s\big|<\infty\quad \Prb\hbox{-a.s.}
%  \end{align*}
%  Thus
%  \begin{equation*}
%    \bar{M}^h_t\leqslant \frac{H}{H_-}\exp\left[\frac{1}{H_-}\int_{(0,t]}\big|\ud L_s^{\A^ch}\big|\right]\left(1\vee(1+B_-)^{-KN_t}\right)<\infty \quad\Prb\hbox{-a.s.}
%  \end{equation*}
%  Then $\Ep[\bar{M}^h_t]<\infty$ for every $t\geqslant 0$. This completes the proof.
\end{proof}

\section{Change of measure}

%The following theorem is proved by \cite[Proposition 3 and 5]{Kunita1963Notes}, see also \cite{Dynkin1965Markov}.
%
%\begin{theorem}\label{thm.Markovian}
%Let $X=\{X_t\}_{t\geqslant 0}$ be a Markov process on $(\Omega,\F,\Prb)$ with right continuous filtration $\F=\{\F_t\}_{t\geqslant 0}$, and $M=\{M_t\}_{t\geqslant 0}$ a positive c\`adl\`ag martingale defined on the filtered probability space $(\Omega,\F,\Prb)$ such that $\Ep[M_t]=1$. We define a new probability measure $\tilde{\Prb}$ by
%$$\frac{\ud\tilde{\Prb}_t}{\ud\Prb_t}=M_t,$$
%where the martingale $M$ is a multiplicative functional. Then on the new filtered probability space $(\Omega,\F,\tilde{\Prb})$, the process $X$ is Markovian.
%\end{theorem}

Consider the general PDMP $X$ from Section 3. Throughout this section $h\in\D^*(\A)$ is a good function. Define a family of probability measures $\{\tilde{\Prb}_t\}_{t\geqslant 0}$ by
\begin{equation}\label{eq.newP}
  \frac{\ud\tilde{\Prb}_t}{\ud\Prb_t}=M^h_t,\quad t\geqslant 0.
\end{equation}
And the \emph{standard set-up} is satisfied, that is, there exists a unique probability measure $\tilde{\Prb}$ such that $\tilde{\Prb}_t=\tilde{\Prb}_{|\F_t}$.
%Thus, following from Theorem \ref{thm.Markovian}, the process $X$ is still a general PDMP on the new filtered probability space $(\Omega,\F,\tilde{\Prb})$.

\begin{theorem}\label{thm.new.Markov&triple}
  Let $X=\{X_t\}_{t\geqslant 0}$ be a general PDMP on $(\Omega,\F,\Prb)$ with measure-valued generator $(\A,\D(\A))$. We define a new probability measure $\tilde{\Prb}$ by \textnormal{(\ref{eq.newP})}. Then on the new probability space $(\Omega,\F,\tilde{\Prb})$, the process $X$ is a general PDMP with the the unchanged SDS $\phi$ and the following conditional hazard function and transition kernel
\begin{align}
  %\tilde{F}(x,\ud t)=&\frac{Qh(\phi_x(t))}{h(\phi_x(t-))+\Delta\A h(\phi_x(t))}M^h(x,t-)F(x,\ud t) \\
  &\tilde{\Lambda}(x,\ud t)=\frac{Qh(\phi_x(t))}{h(\phi_x(t-))+\Delta\A h(\phi_x(t))} \Lambda(x,\ud t), \label{eq.new.Lambda}\\
  &\tilde{Q}(\phi_x(t),\ud y)=\frac{h(y)}{Qh(\phi_x(t))}Q(\phi_x(t),\ud y),\label{eq.new.Q}
\end{align}
for $x\in E$, $t\in\mathcal{I}_x$, where $Qh(x)=\int_E h(y)Q(x,\ud y)$ for $x\in\bar{E}$.
\end{theorem}

\begin{proof}
  By \cite[Lemma 8.6.2]{Oksendal2003Stochastic} we get, for any $\F$-stopping time $T$ and $t>0$,
\begin{align*}
   &\tilde{\Ep}[f(X_{T+t})|\F_T]\\
  =&\frac{\Ep[f(X_{T+t})M^h_{T+t}|\F_T]}{\Ep[M^h_{T+t}|\F_T]}
  =\frac{\Ep[f(X_{T+t})M^h_{T+t}|\F_T]}{M^h_T}\\
  =&\Ep\left[f(X_{T+t})\frac{h(X_{T+t})}{h(X_T)} \exp\!\left[-\int_{(T,T+t]}\!\frac{\ud L(\A^ch)_s}{h(X_{s-})}\right] \!\!\prod_{T<s\leqslant T+t}\!\!\left[1+\frac{\Delta\A h(X_s^-)}{h(X_{s-})}\right]^{-1}\Big|\F_T\right]\\
  =&\Ep\left[f(X_{T+t})\frac{h(X_{T+t})}{h(X_T)} \exp\!\left[-\int_{(T,T+t]}\!\frac{\ud L(\A^ch)_s}{h(X_{s-})}\right] \!\!\prod_{T<s\leqslant T+t}\!\!\left[1+\frac{\Delta\A h(X_s^-)}{h(X_{s-})}\right]^{-1}\Big|X_T\right]\\
  =&\tilde{\Ep}[f(X_{T+t})|X_T].
\end{align*}
  Thus, the process $X$ has the strong Markov property on $(\Omega,\F,\tilde{\Prb})$.

Now we denote
\begin{equation*}
    m^h(x,t)=\frac{h(\phi_x(t))}{h(x)}g^h(x,t),
\end{equation*}
where
\begin{equation*}
  g^h(x,t)=\exp\left[-\int_{(0,t]}\frac{\A^ch(x,\ud s)}{h(\phi_x(s-))}\right] \prod_{0<s\leqslant t}\left[1+\frac{\Delta\A h(\phi_x(s))}{h(\phi_x(s-))}\right]^{-1}
\end{equation*}
for $x\in E$, $t\in \mathcal{I}_x$.
By (\ref{eq.newP}), on $(\Omega,\F,\tilde{\Prb})$, we have the following conditional survival function and transition kernel for $\{(\tau_n,X_{\tau_n})\}_{n\geqslant 0}$
\begin{align*}
  \tilde{F}(x,t)=&\tilde{\Ep}_x[\bbbone_{\{\tau_1>t\}}]=\Ep_x[M^h_t\bbbone_{\{\tau_1>t\}}]
%  =&F(x,t)\,M^h_t\big|_{t<\tau_1} \\
%  =& M^h(x,t)F(x,t) \\
  =m^h(x,t)F(x,t),\\
  \tilde{G}(x,\ud t,\ud y)
  =& \tilde{\Ep}_x[\bbbone_{\{\tau_1\in\ud t\}}\bbbone_{\{X_{\tau_1}\in\ud y\}}]
  = \Ep_x[E_t^h\bbbone_{\{\tau_1\in\ud t\}}\bbbone_{\{X_{\tau_1}\in\ud y\}}] \\
%  =& F(x,\ud t)Q(\phi_x(t),\ud y)\,E_t^h\big|_{X_t=y,t=\tau_1} \\
  =& \frac{h(y)}{h(x)}g^h(x,t)F(x,\ud t)Q(\phi_x(t)),\ud y) .
\end{align*}
For any $x\in E$, $t\in\mathcal{I}_x$,
$$g^h(x,\ud t)=-\frac{g^h(x,t-)}{h(\phi_x(t-))+\Delta\A h(\phi_x(t))}\A h(x,\ud t).$$
Then, by the formula of integration by parts, we have
\begin{align*}
  m^h(x,\ud t)
  =& \frac{1}{h(x)}\big\{g^h(x,t-)\ud h(\phi_x(t))+h(\phi_x(t))g^h(x,\ud t)\big\} \\
  =& \frac{g^h(x,t-)}{h(x)}\left\{\ud h(\phi_x(t))-\frac{h(\phi_x(t))\A h(x,\ud t)} {h(\phi_x(t-))+\Delta\A h(\phi_x(t))}\right\} \\
  =& \frac{g^h(x,t-)}{h(x)}\left\{\ud h(\phi_x(t))+\frac{\A h(x,t-)\ud h(\phi_x(t))-\ud\big(h(\phi_x(t))\A h(x,t)\big)}{h(\phi_x(t-))+\Delta\A h(\phi_x(t))}\right\} \\
  =& \frac{g^h(x,t-)}{h(x)}\frac{h(\phi_x(t-))\ud h(\phi_x(t))+\A h(x,t)\ud h(\phi_x(t))-\ud\big(h(\phi_x(t))\A h(x,t)\big)}{h(\phi_x(t-))+\Delta\A h(\phi_x(t))} \\
  =& \frac{g^h(x,t-)}{h(x)}\frac{h(\phi_x(t-))\ud h(\phi_x(t))-h(\phi_x(t-))\A h(x,\ud t)}{h(\phi_x(t-))+\Delta\A h(\phi_x(t))} \\
  =& -m^h(x,t-)\frac{Qh(\phi_x(t))-h(\phi_x(t))}{h(\phi_x(t-))+\Delta\A h(\phi_x(t))} \Lambda(x,\ud t),
\end{align*}
and
\begin{align*}
  \tilde{F}(x,\ud t)
  =& m^h(x,t)F(x,\ud t)-F(x,t-)m^h(x,\ud t) \\
  %=& M^h(x,t-)\left[\frac{h(\phi_x(t))}{h(\phi_x(t-))}\frac{h(\phi_x(t-))}{h(\phi_x(t-)) +\Delta\A h(\phi_x(t))}+\frac{Qh(\phi_x(t))-h(\phi_x(t))}{h(\phi_x(t-))+\Delta\A h(\phi_x(t))}\right]F(x,\ud t) \\
  =& m^h(x,t-)\frac{h(\phi_x(t))F(x,\ud t)+F(x,t-)[Qh(\phi_x(t))-h(\phi_x(t))]\Lambda(x,\ud t)}{h(\phi_x(t-))+\Delta\A h(\phi_x(t))} \\
  =& m^h(x,t-)\frac{Qh(\phi_x(t))}{h(\phi_x(t-))+\Delta\A h(\phi_x(t))}F(x,\ud t).
\end{align*}
Thus
\begin{align*}
  &\tilde{\Lambda}(x,\ud t)=\frac{\tilde{F}(x,\ud t)}{\tilde{F}(x,s-)}
  =\frac{Qh(\phi_x(t))}{h(\phi_x(t-))+\Delta\A h(\phi_x(t))}\Lambda(x,\ud t),\\
  &\tilde{Q}(\phi_x(t),\ud y) = \frac{\tilde{G}(x,\ud t,\ud y)}{\tilde{F}(x,\ud t)}
  = \frac{h(y)}{Qh(\phi_x(t))}Q(\phi_x(t),\ud y).
\end{align*}
The theorem is proved.
\end{proof}

By (\ref{eq.new.Lambda}), we notice that $\tilde{\Lambda}(x,\cdot)$ is absolutely continuous with respect to $\Lambda(x,\cdot)$. Thus, we have $J_{\tilde{\Lambda}}=J_\Lambda$.

\begin{theorem}\label{thm.new.generator}
  Let $X=\{X_t\}_{t\geqslant 0}$ be a general PDMP on $(\Omega,\F,\Prb)$ with measure-valued generator $(\A,\D(\A))$. Probability measure $\tilde{\Prb}$ is defined by \textnormal{(\ref{eq.newP})}. Then on the new probability space $(\Omega,\F,\tilde{\Prb})$, the measure-valued generator of $X$ is
  \begin{equation}\label{eq.new.Af}
    \tilde{\A}f(x,\ud t)=Df(x,\ud t)+\Lambda(x,\ud t) \int_E\frac{\big[f(y)-f(\phi_x(t))\big]h(y)}{h(\phi_x(t-))+\Delta\A h(\phi_x(t))}Q(\phi_x(t),\ud y)
  \end{equation}
 with the domain $\D(\tilde{\A})$ which contains all the functions $f\in\mathcal{M}(\bar{E})$ with locally path-finite-variation such that for any $x\in E$ and $t\in\mathcal{I}_x$,
 \begin{equation}\label{eq.new.D(A)}
   \int_{(0,t]}\int_E\frac{|f(y)-f(\phi_x(s))|h(y)}{h(\phi_x(s-))+\Delta\A h(\phi_x(s))}Q(\phi_x(s),\ud y)\Lambda(x,\ud s)<\infty.
 \end{equation}
\end{theorem}

\begin{proof}
  By the form of the measure-valued generator and its domain, following from Theorem \ref{thm.new.Markov&triple}, the conclusion can be get directly.
\end{proof}

\begin{corollary}\label{cor.new.Af-1&2}
  Note that the new measure-valued generator \textnormal{(\ref{eq.new.Af})} can be rewritten as
\begin{equation}\label{eq.new.Af-1}
  \tilde{\A}f(x,\ud t)
  =\frac{\A(fh)(x,\ud t)-f(\phi_x(t-))\A h(x,\ud t)}{h(\phi_x(t-))+\Delta\A h(\phi_x(t))} ,
\end{equation}
or using the \emph{op\'erateur carr\'e du champ}
\begin{equation}\label{eq.new.Af-2}
  \tilde{\A}f(x,\ud t)=\A f(x,\ud t)
+\frac{\langle f,h\rangle_{\A}(x,\ud t)}{h(\phi_x(t-))+\Delta\A h(\phi_x(t))},
\end{equation}
where $\langle f,h\rangle_{\A}$ is also an additive functional of the SDS $\phi$ defined by
\begin{align*}%\label{eq.<fh>A}
  \langle f,h\rangle_{\A}(x,\ud t)
=&\A(fh)(x,\ud t)-f(\phi_x(t-))\A h(x,\ud t)-h(\phi_x(t-))\A f(x,\ud t)\\
&-\ud [\A f(x,t),\A h(x,t)]_t,
\end{align*}
and
\begin{equation*}
  [\A f(x,t),\A h(x,t)]_t=\sum_{0<s\leqslant t}\Delta\A f(\phi_x(s))\,\Delta\A h(\phi_x(s)).
\end{equation*}
\end{corollary}

\begin{proof}
Let $g\in \mathfrak{A}_\phi$ defined by
\begin{equation*}
  g(x,\ud t)=\A f(x,\ud t)
+\frac{\langle f,h\rangle_{\A}(x,\ud t)}{h(\phi_x(t-))+\Delta\A h(\phi_x(t))}
\end{equation*}
for $x\in E$, $t\in\mathcal{I}_x$.
Notice that, by the formula of integration by parts
\begin{align*}
  \ud[\A f(x,t),\A h(x,t)]_t=&\ud\big(\A f(x,t)\A h(x,t)\big)-\A f(x,t-)\A h(x,\ud t)\\
  &-\A h(x,t-)\A f(x,\ud t).
\end{align*}
Thus
\begin{align*}
  g(x,\ud t)=&\frac{1}{h(\phi_x(t-))+\Delta\A h(\phi_x(t))}
  \Big\{\Delta\A h(\phi_x(t))\A f(x,\ud t)\\
  &+\A(fh)(x,\ud t)-f(\phi_x(t-))\A h(x,\ud t)-\ud\big(\A f(x,t)\A h(x,t)\big)\\
  &+\A f(x,t-)\A h(x,\ud t)+\A h(x,t-)\A f(x,\ud t)\Big\}\\
  =&\frac{1}{h(\phi_x(t-))+\Delta\A h(\phi_x(t))} \Big\{\A(fh)(x,\ud t)-f(\phi_x(t-))\A h(x,\ud t)\\
  &-\ud\big(\A f(x,t)\A h(x,t)\big)+\A f(x,t-)\A h(x,\ud t)+\A h(x,t)\A f(x,\ud t)\Big\}\\
  =&\frac{\A(fh)(x,\ud t)-f(\phi_x(t-))\A h(x,\ud t)}{h(\phi_x(t-))+\Delta\A h(\phi_x(t))}.
\end{align*}
Furthermore, note that
\begin{align*}
   &f(\phi_x(t-))\A h(x,\ud t)\\
  =&\ud\big(f(\phi_x(t))\A h(x,t)\big)-\A h(x,t)\ud f(\phi_x(t))\\
  =&\ud\big(f(\phi_x(t))\A h(x,t)\big)-\A h(x,t-)\ud f(\phi_x(t))-\Delta\A h(\phi_x(t))\ud f(\phi_x(t))\\
  =&f(\phi_x(t))\A h(x,\ud t)-\Delta\A h(\phi_x(t))Df(x,\ud t),
\end{align*}
then
\begin{align*}
  &\A(fh)(x,\ud t)-f(\phi_x(t-))\A h(x,\ud t)\\
  =&D(fh)(x,\ud t)+\Lambda(x,\ud t)\left[\int_Ef(y)h(y)Q(\phi_x(t)\ud y) -f(\phi_x(t))h(\phi_x(t))\right]\\
  &-f(\phi_x(t))\Big[Dh(x,\ud t)+\Lambda(x,\ud t)\big[Qh(\phi_x(t))-h(\phi_x(t))\big]\Big] \\ &-\Delta\A h(\phi_x(t))Df(x,\ud t)\\
  =&D(fh)(x,\ud t)-f(\phi_x(t))Dh(x,\ud t)-\Delta\A h(\phi_x(t))Df(x,\ud t)\\
  &+\Lambda(x,\ud t)\Big[\int_Ef(y)h(y)Q(\phi_x(t)\ud y)-f(\phi_x(t-))Qh(\phi_x(t))\Big]\\
  =&h(\phi_x(t-))Df(x,\ud t)-\Delta\A h(\phi_x(t))Df(x,\ud t)\\
  &+\Lambda(x,\ud t)\int_E\big[f(y)-f(\phi_x(t))\big]h(y)Q(\phi_x(t),\ud y).
\end{align*}
Thus we have $g=\tilde{\A} f$, which completes the proof.
\end{proof}

\begin{theorem}
  Let $X$ be a general PDMP on $(\Omega,\F,\Prb)$ with $(\A,\D(\A))$. Probability measure $\tilde{\Prb}$ is defined by \textnormal{(\ref{eq.newP})}. Then
  \begin{equation}\label{eq.reverse}
    \frac{\ud\Prb_t}{\ud\tilde{\Prb}_t}=\tilde{M}_t^{h^{-1}}
    =\frac{h^{-1}(X_t)}{h^{-1}(X_0)}\left[\sexp\left(\int_{(0,t]}\frac{\ud L(\tilde{\A}h^{-1})_s}{h^{-1}(X_{s-})}\right)\right]^{-1},\quad t\geqslant 0.
  \end{equation}
\end{theorem}

\begin{proof}
  By (\ref{eq.newP}), we have
  $$\frac{\ud\Prb_t}{\ud\tilde{\Prb}_t}=(M^h_t)^{-1}.$$
  So we only need to prove that $\tilde{M}_t^{h^{-1}}=(M^h_t)^{-1}$.

  Following from (\ref{eq.new.Af-1}), we have
%  \begin{equation*}
%    \tilde{\A}h^{1}(x,\ud t)=-\frac{h^{-1}(\phi_x(t-))\A h(x,\ud t)}{h(\phi_x(t-))+\Delta\A h(\phi_x(t))}.
%  \end{equation*}
%  Thus,
  \begin{equation*}
    \frac{\tilde{\A}h^{-1}(x,\ud t)}{h^{-1}(\phi_x(t-))}=-\frac{\A h(x,\ud t)}{h(\phi_x(t-))+\Delta\A h(\phi_x(t))}.
  \end{equation*}
  Then
  \begin{align*}
    \frac{\tilde{\A}^ch^{-1}(x,\ud t)}{h^{-1}(\phi_x(t-))}=&-\frac{\A^c h(x,\ud t)}{h(\phi_x(t-))},\\
    1+\frac{\Delta\tilde{\A}h^{-1}(\phi_x(t))}{h^{-1}(\phi_x(t-))}=&\frac{h(\phi_x(t-))}{h(\phi_x(t-))+\Delta\A h(\phi_x(t))}.
  \end{align*}
  Hence, we get $\tilde{M}_t^{h^{-1}}=(M^h_t)^{-1}$.
\end{proof}

%\section*{References}



  \bibliographystyle{elsart-num-sort}
  \bibliography{PDMPsbib}




%%\bibitem[ ()]{}

%%\end{thebibliography}
\end{document}